\documentclass{amsart}

\providecommand{\dotdiv}{
  \mathbin{
    \vphantom{+}
    \text{
      \mathsurround=0pt 
      \ooalign{
        \noalign{\kern-.35ex}
        \hidewidth$\smash{\cdot}$\hidewidth\cr 
        \noalign{\kern.35ex}
        $-$\cr 
      }%
    }%
  }%
}

\let\minus=\dotdiv

\def\ov{\overline}
\def\E{\mathbf{E}}
\def\A{\mathbf{A}}

\def\Lu{\mathbf{C}_3}
\def\He{\mathbf{H}_3}

\newtheorem*{question*}{Question}
\newtheorem{theorem}{Theorem}
\newtheorem{lemma}{Lemma}

\author{Tomasz Kowalski}
\title{The bottom of the lattice of BCK-varieties}

\begin{document}

\maketitle

\begin{abstract}
Confirming a conjecture of Pałasiński and Wroński, we show that the bottom otf he
lattice of subvarieties of BCK is Y-shaped.
\end{abstract}

\section{Introduction}

The class of BCK-algebras, introduced in Imai \& Iseki~\cite{II66} 
as an algebraic counterpart of BCK-logic 
and extensively studied ever since, can be viewed (dually)
as the class of all algebras $\A = \langle A; \minus, 0\rangle$ of type
$\langle 2,0\rangle$ such that $\A$ satisfies the following
identities:

\begin{enumerate}
\item $((x \minus y) \minus (x \minus z)) \minus (z \minus y) = 0$
\item $x \minus 0 = x$
\item $0 \minus x = 0$,
\end{enumerate}
and the quasi-identity:
\begin{enumerate}
\setcounter{enumi}{3}
\item $x \minus y = 0 = y \minus x \ \Rightarrow \  x = y$.
\end{enumerate}

The universe of any BCK-algebra is partially ordered by the relation 
$a\leq b$ iff $a\minus b = 0$.

Hereafter we will omit the $\minus$ sign, as in the following:
\begin{enumerate}
\setcounter{enumi}{4}
\item $xx = 0$, 
\item $x(xy)\leq y$, i.e., $(x(xy))y = 0$,
\item $(xy)z = (xz)y$.
\end{enumerate}
which are true in all BCK-algebras, and are mentioned here to facilitate 
reading the calculations to come. 

It is evident from the definition that the class of all BCK-algebras 
is a quasivariety. Yet it is not a variety, nor does the largest 
subvariety of BCK exist---as shown in Wroński~\cite{Wro83}, 
Wroński \& Kabziński~\cite{WK84}, respectively. 
Several subvarieties of BCK have been isolated and thoroughly
investigated, which has led to a substantial body of results 
(cf. e.g., Blok \& Raftery~\cite{BR95}). Nevertheless, 
the question of describing the bottom of the lattice of BCK-varieties, 
raised in Pałasiński \& Wroński~\cite{PW86}, has been remaining open.

Let us now recall some basic concepts that will be 
of use in the sequel.

An {\it ideal\/} $I$ of a BCK-algebra $\A$ is a subset of $A$, such that:
$0 \in I$; and whenever $b\in I, a\minus b\in I$, then $a\in I$ as well. 
By a {\it BCK-congruence\/} of $\A$ we mean a congruence $\Phi$ such that
$\A/_{\Phi}$ is a BCK-algebra. 
If $\Theta$ is any congruence of $\A$, then its equivalence class of $0$, 
$[0]_{\Theta}$, is an ideal of $\A$. Conversely, for each ideal $I$ of $\A$
there is a congruence $\Theta$ with $[0]_{\Theta} = I$. In general, 
this congruence need not be unique, and it need not be a BCK-congruence, 
either. However, the congruence $\Theta_I$, defined by:
$$(a,b)\in \Theta_I {\rm \ iff\ } a\minus b \in I {\rm \ and\ } 
b\minus a\in I$$
turns out to be the largest congruence with $[0]_{\Theta} = I$, and, 
at the same time, the unique BCK-congruence with this property.
   
By $I(a)$ we will mean the ideal generated by an element $a \in A$,
we have $b\in I(a)$ iff there is an $n\leq\omega$ with 
$ab^n = 0$, where the exprssion $ab^n$ is meant to abbreviate
(here, and later on) $(a\underbrace{b)b\dots b}_{n\ {\rm times}}$.

A BCK-algebra $\A$ is subdirectly irreducible if and only if
it has the smallest nontrivial ideal $I$,
$\A$ is simple if and only if $I = A$.

As it is easy to verify, there exist, up to isomorphism: 
exactly one two-element BCK-chain 
$\mathbf{C}_2 = \langle \{0,1\}; \minus, 0\rangle$; 
precisely two three-element BCK-chains, 
$\Lu = \langle \{0,{1\over 2}, 1\}; \minus, 0\rangle$, 
with $1\minus {1\over 2} = {1\over 2}$; and 
$\He = \langle \{0,{1\over 2}, 1\}; \minus, 0\rangle$, 
with $1\minus {1\over 2} = 1$. $\mathbf{C}_2$ is (dually) 
isomorphic to the implication reduct of the 
two-element Boolean algebra, while $\Lu$ and $\He$ are
(dually) isomorphic to the implicational reducts of the three-element
Łukasiewicz algebra, and the three-element totally ordered Heyting algebra, 
respectively.

Since every non-trivial BCK-algebra contains a subalgebra isomorphic to 
$\mathbf{C}_2$, the variety $\mathcal{C}_2$ generated by $\mathbf{C}_2$
is the unique atom of the lattice of BCK-varieties. To the description
of the next level of this lattice, the following question is crucial.

\begin{question*}[Question 2 in~\cite{PW86}]  Is it true that for every variety
$\mathcal{V}$ of BCK-algebras either $\mathcal{V}$ is contained in $\mathcal{C}_2$ or
$\{\Lu, \He\} \cap \mathcal{V}$ is nonempty?
\end{question*}

We will show that the answer to the above question is positive.

\section{Answering the question}

Consider a si BCK-algebra $\A$ nonisomorphic to $\mathbf{C}_2$.

\begin{lemma}\label{fct:zero}
If $\A$ has an atom, then $\Lu\leq\A$ or $\He\leq\A$.
\end{lemma}

\begin{proof}
If $\A$ has an atom $a$, then $a$ must be 
unique, smaller than any other non-zero element,  
and must belong to the smallest ideal of $\A$. 
The reader is asked to verify that this is indeed so. 

Let then $b\in \A$ with $b\neq a$ and consider the element 
$(ba)((ba)a)$. Since,  
$(ba)((ba)a) \leq a$ and $a$ is an atom, we have only two possibilities: 

\begin{itemize}
\item[(i)] $(ba)((ba)a) = a$, thus, $\left(b\bigl((ba)a\bigr)\right)a = a$, in
which case, putting  $1 = b((ba)a)$ and $\frac{1}{2} = a$, we get   
$\Lu = \langle\{1, \frac{1}{2}, 0\}; \minus, 0\rangle  \leq\A$; 
\item[(ii)] $(ba)((ba)a) = 0$, thus, $ba \leq (ba)a$, hence 
$ba = (ba)a$, so putting $1 = ba$ and $\frac{1}{2} = a$, we obtain  
$\He = \langle\{1, \frac{1}{2}, 0\}; \minus, 0\rangle \leq\A$.\qedhere
\end{itemize}
\end{proof}

Let us now assume that $\A$ contains no subalgebra isomorphic to
either $\Lu$ or $\He$.
It follows then, by Lemma~\ref{fct:zero}, that the smallest ideal of $\A$ 
(indeed, any subalgebra of $\A$ with more than two elements) 
must be infinite. Let us choose a nonzero element $a$ from the smallest 
ideal of $A$.
The set $\{x\in A: x\leq a\}$ is a subuniverse of $\A$ and the subalgebra 
$\A|_a$, with this universe, is an infinite, simple subalgebra of $\A$
with the greatest element. Let us denote this algebra by $\E$, 
and its greatest element by~$1$.

By the fact the algebra $\E$ is simple we get that, for any elements 
$a, b \in E$ there is a $k < \omega$ such that $ab^k = 0$.
The smallest such $k$ we will call the
{\it height\/} of $a$ {\it relative to\/} $b$. 
Observe the following:

\begin{lemma}\label{fct:one}
If, for an element $a$ of $E$, there is an upper bound for 
its relative heights, then $\E$ has an atom. 
\end{lemma}

\begin{proof}
Assume there is an upper bound for relative heights 
of an element $a$ of $\E$, i.e. a $k<\omega$ such that for every $x\in E$, 
we have $ax^k=0$. Let $n$ be the smallest such, and let's choose $b\in E$ 
with $ab^n=0$ and $ab^{n-1} \neq 0$. We will show $ab^{n-1}$ is an atom.

Suppose it is not. Then, there is a $c\in E$ with $0<c<ab^{n-1}$.
Since $ab^n=0$, we have $ab^{n-1}\leq b$ and thus $c\leq b$ as well.
Therefore, $ac^{n-1} \geq ab^{n-1}$, but from the fact than $n$ is the
greatest possible relative height of $a$, we obtain 
$ac^n=0$, i.e. $ac^{n-1}\leq c$. Together, it gives
$ab^{n-1} \leq ac^{n-1} \leq c$, contradicting the assumption. Hence,
$ab^{n-1}$ is an atom.        
\end{proof}

Now, to retain the assumption that neither $\Lu \leq\E$ nor $\He\leq\E$ 
we must also assume that there is no upper bound for relative heights 
of elements of $\E$.  

Let us take an $e\in E$ with $0<e<1$. Thus, 
there is an $1<n<\omega$ such that $1e^n=0$ and $1e^{n-1}>0$.
For $i=0,\dots,n-1$ define:
\begin{align*}
P_i &= 1e^i\\
Q_i &= P_i(P_{i+1}(\dots(P_{n-2}P_{n-1}))\dots).
\end{align*}

Notice that $P_0 =1$, $P_{n-1}e=0$, $P_{n-1}>0$, 
and $Q_i=P_iQ_{i+1}$.

\begin{lemma}\label{fct:two}
For $i=1,\dots,n-1$, $Q_iQ_{i-1}=0$. 
\end{lemma}

\begin{proof} We proceed by downward induction\footnote{I am
    indebted to Andrzej Wroński for presenting my long calculation in
    this neat way.} on $i$ with step 2.   

\noindent  
Base step, for $i=n-1$ and $i=n-2$:
$$
Q_{n-1}Q_{n-2} = P_{n-1}(P_{n-2}P_{n-1})
= (P_{n-2}(P_{n-2}P_{n-1}))e \leq P_{n-1}e = 0,
$$
since $1e^n =0$. Next,
\begin{align*}
Q_{n-2}Q_{n-3} &= (P_{n-2}P_{n-1})Q_{n-3} =
                 ((P_{n-3}Q_{n-3})P_{n-1})e\\
  &= \bigl((P_{n-3}(P_{n-3}Q_{n-2}))P_{n-1}\bigr)e \leq (Q_{n-2}P_{n-1})e\\
 &= ((P_{n-2}P_{n-1})P_{n-1})e = 
(P_{n-1}P_{n-1})P_{n-1} = 0.
\end{align*}
   
\noindent
Inductive step, two levels down:
\begin{align*}
Q_{i-2}Q_{i-3} &= 
                 (P_{i-2}(P_{i-1}Q_i))(P_{i-3}(P_{i-2}Q_{i-1}))\\
               &=\bigl((P_{i-3}(P_{i-1}Q_i))(P_{i-3}(P_{i-2}Q_{i-1}))\bigr)e\\
               &=\bigl((P_{i-3}(P_{i-3}(P_{i-2}Q_{i-1})))(P_{i-1}Q_i)\bigr)e\\
               &\leq ((P_{i-2}Q_{i-1})(P_{i-1}Q_i))e\\
               &= (P_{i-1}Q_{i-1})(P_{i-1}Q_i) \leq Q_iQ_{i-1} = 0,
\end{align*}
where the last equality follows by inductive hypothesis.
\end{proof}

With the help of Lemma~\ref{fct:two}, we easily obtain: 
\begin{lemma}\label{fct:three}
The following hold:
\begin{itemize}
\item[(i)] $Q_1Q_0=0$, in other words $Q_1(1Q_1)=0$;
\item[(ii)] $Q_0Q_1\leq e$, in other words $(1Q_1)Q_1\leq e$.
\end{itemize}
\end{lemma}

\begin{proof}
The first is a particular case of Lemma~\ref{fct:two}; as for the second:
$(Q_0Q_1)e = ((1Q_1)Q_1)e = (P_1Q_1)Q_1 = 
(P_1(P_1Q_1))Q_1 \leq Q_2Q_1 = 0$, by Lemma~\ref{fct:two}.
\end{proof}

Consider a descending sequence $\ov{e} = (e_n)_{n<\omega}$ 
of elements of $E$, converging to 0. Such a sequence exists, as $\E$
has no atoms. Of course, $\ov{e}$ is an element of $\E^\omega$. Consider
$I(\ov{e})$, the ideal generated by $\ov{e}$ in $\E^\omega$. Notice that:

\begin{lemma}\label{fct:four}
No constant sequence belongs to $I(\ov{e})$.
\end{lemma}

\begin{proof}
Suppose that $\ov{a} = \langle a,a,\dots,a,\dots\rangle$ 
belongs to $I(\ov{e})$. This means, there is an $n<\omega$ such that
$\ov{a}\ov{e}^n = 0$, i.e. $\forall i<\omega: a{e_i}^n = 0$. Since $\ov{e}$
converges to $0$, we have 
$\forall d\in E, d>0 \exists i<\omega: e_i\leq d$.
Therefore, $ae_i\geq ad$, and further $a{e_i}^k\geq ad^k$,
for any $k>0$. Hence, in particular, $0 = a{e_i}^n\geq ad^n$, 
and thus $n$ is an upper bound for relative heights of $a$ 
in $\E$. This contradicts the assumption of there being no such a bound.
\end{proof}

Take now the largest congruence $\Theta$ on $\E^\omega$ with 
$[0]_\Theta = I(\ov{e})$. By Lemma~\ref{fct:three}, this congruence is neither trivial 
nor full, and, 
moreover, $|{\E^\omega}/_\Theta| \geq |\E|$. 

For any $e_i$, let us write $q_i$ for the element $Q_1$, defined as before, 
for this particular $e_i$. Consider the sequence
$\ov{q} = \langle q_0, q_1, \dots, q_i, \dots\rangle$
of elements of $E$, and let $q = {\ov{q}}/_\Theta$ so that 
$q\in E^\omega/\Theta$. 

\begin{lemma}\label{fct:five}
The quotient algebra ${E^\omega}/_\Theta$ verifies $1q = q$.  
Hence $\Lu\in \mathcal{V}(\E)$.
\end{lemma}

\begin{proof}
By Lemma~\ref{fct:three}, we have
$\ov{q}(1\ov{q}) = 0 \in I(\ov{e})$ and 
$(1\ov{q})\ov{q}\leq \ov{e} \in I(\ov{e})$, as well. 
Thus, the first part follows by the definition of 
the congruence $\Theta$.

For the second part, observe that
$\{1,q,0\} \subseteq {E^\omega}/_\Theta$ is a subuniverse of
${\E^\omega}/_\Theta$ and the algebra with this universe is isomorphic to
$\Lu$. 
\end{proof}

\begin{theorem}
If a variety $\mathcal{V}$ of BCK-algebras is not contained in
$\mathcal{C}_2$, then $\{\Lu, \He\} \cap \mathcal{V}$ is nonempty.
\end{theorem}

\begin{proof}
Since $\mathcal{V} \not\subseteq \mathcal{C}_2$, there is a subdirectly 
irreducible algebra $\A\in \mathcal{V}$  with more than two elements. If
$\A$ has an atom, the result follows by Lemma~\ref{fct:zero}.
If $\A$ has no atoms, 
then it contains an infinite simple algebra $\E$ without an upper bound
for relative heights of its elements. Then the result follows by
Lemma~\ref{fct:five}.  
\end{proof} 

\bibliographystyle{plain}
\bibliography{BCK-bottom}

\end{document}